\documentclass{article}
\usepackage[a4paper]{geometry}
\usepackage{authblk}
\usepackage{graphicx} % Required for inserting images
\usepackage[ruled,vlined]{algorithm2e}
\usepackage{amsmath,latexsym,amssymb,xcolor}
\usepackage{enumitem}

\def\N{\mathbb{N}}

\def\R{\mathbb{R}}
 
\def\E{\mathbb{E}}

\def\Id{{\rm I}_d}
\def \T{\intercal}

\usepackage{amsthm}
\newtheorem{theorem}{Theorem}[section]
\newtheorem{lemma}[theorem]{Lemma}
\newtheorem{corollary}[theorem]{Corollary}
\newtheorem{proposition}[theorem]{Proposition}

\theoremstyle{definition}
\newtheorem{definition}[theorem]{Definition}

\theoremstyle{remark}
\newtheorem{remark}[theorem]{Remark}
\input{mathsymbols}
\makeatother
\usepackage[square,numbers]{natbib} 

\title{\vspace{-2.2em}
{Linearization of Monge--Ampère Equations and Statistical Applications}}
\date{\today}
\author{
  Alberto Gonz{\'a}lez-Sanz%
  \thanks{Department of Statistics, Columbia University, ag4855@columbia.edu.},  \quad\quad\quad
  Shunan Sheng%
  \thanks{Department of Statistics, Columbia University, ss6574@columbia.edu.}
  }

\begin{document}
\maketitle

\begin{abstract}
Optimal transport has found numerous applications across data science, many of which require differentiating the optimal transport map with respect to the underlying probability densities in the Fréchet sense. In this work, we show that when the reference measure $Q$ is sufficiently regular in space and the curve of target measures $\{P_t\}_{t\in I}$ is both spatially regular and $\mathcal{C}^1$ in time, then the associated curve of optimal transport maps $\{\nabla \phi_t\}_{t\in I}$ pushing $Q$ toward $P_t$ is itself a $\mathcal{C}^1$ curve. Moreover, we identify its time derivative as the solution to the \emph{linearized Monge--Ampère equation}, a second-order elliptic PDE with strictly oblique boundary conditions and a vanishing zero-order term. Our proof relies on applying the implicit function theorem to the Monge--Ampère equation with natural boundary conditions. As consequences, we establish regularity of the transport-based quantile regressor with respect to the covariates and derive a central limit theorem for smooth optimal transport maps.
\end{abstract}

\maketitle

\section{Introduction}
The optimal transport problem has become increasingly popular in data science, appearing frequently in both statistics~\cite{carlier2016vector,delBarrioQuantile2024,HallinBarrio2021,hosseini2024conditionaloptimaltransportfunction} and machine learning, including Bayesian inverse problems~\cite{hosseini2024conditionaloptimaltransportfunction,kerrigan2024dynamicconditionaloptimaltransport}, generative modeling~\cite{Wgans,albergo2023stochastic}, and deep learning~\cite{castin2024smoothattention}, among others. 
In many of these applications (see, e.g.,~\cite{hamm2024manifoldlearningwassersteinspace,hosseini2024conditionaloptimaltransportfunction,manole2023central,sheng2025stability,sheng2025theory}), an important and persistent challenge is to quantify how optimal transport maps change when the source measure is fixed and the target measure is perturbed. In particular, let $I \subset \mathbb{R}$ be an open interval, and let $\{P_t\}_{t\in I}$ denote a family of probability measures with densities $\{p_t\}_{t\in I}$ and supports $\{\Omega_t\}_{t\in I}$. 
The present work identifies sufficient conditions on $\{P_t\}_{t\in I}$ and a reference measure $Q \in \mathcal{P}(\Omega)$ guaranteeing that the curve of optimal transport maps $\{\nabla \phi_t\}_{t\in I} \subset \CC^{1,\alpha}(\Omega)$, pushing $Q$ forward to $P_t$, solving
\begin{equation}\label{eq:MongeAmpere}
    \begin{aligned}
        \det(D^2\phi_t) &= \frac{q}{p_t(\nabla\phi_t)} \qquad \text{in } \Omega,\\
        \nabla\phi_t(\Omega) &= \Omega_t,
    \end{aligned}
\end{equation}
is a $\CC^1$ curve in time.

Understanding the time regularity of optimal transport maps under perturbations provides tools for establishing central limit theorems~\cite{del2025distributional,manole2023central} and for analyzing the statistical properties of plug-in estimators~\cite{Manole2024Plugin,balakrishnan2025stability,balakrishnan2025statistical}. 
Moreover, such regularity can be used to study the smoothness of velocity fields in manifold learning problems in the Wasserstein space~\cite{hamm2024manifoldlearningwassersteinspace}.

Although the spatial regularity of $(t,x)\mapsto \phi(t,x):=\phi_t(x)$ is well understood from Caffarelli's regularity theory (see~\cite{Caffarelli1990,caffarelli1992boundary,delBarrioAdvNon2024,figalli2017monge,Figalli2018}), the time regularity is subtler.  
In~\cite[Proposition~4.1]{loeper2005regularity}, the author shows that if the supports $\{\Omega_t\}_{t\in I}$ are constant in $t$ and both the densities and supports are $\CC^\infty$, then $\phi$ is differentiable in $t$. 
As in the present work, the proof relies on an application of the implicit function theorem. 
Furthermore, in~\cite{loeper2006fully}, the same author uses this result to establish the existence and uniqueness of classical solutions to the semigeostrophic equations on a three-dimensional flat torus.

More recent results in~\cite{manole2023central} also focus on the setting where both source and target measures are supported on the flat torus $\mathbb{T}^d=\mathbb{R}^d/\mathbb{Z}^d$, and the transport maps are periodic (see~\cite{CORDEROERAUSQUIN1999199}). 
Their analysis relies crucially on the periodic boundary conditions of the Monge--Ampère equation on the torus, and therefore does not extend to the setting of interest here, particularly when the supports $\{\Omega_t\}_{t\in I}$ vary in time.

A related but distinct line of work concerns stability of optimal transport maps, with extensive literature on both qualitative and quantitative aspects.  
For example, qualitative results include: Theorem~4.2 in~\cite{philippis2013regularity}, which states that if $p_{t'} \to p_t$ in $L^1(\Omega')$, then $\nabla\phi_{t'}\to \nabla\phi_t$ in $W^{1,\gamma}_{\mathrm{loc}}(\Omega)$ for some $\gamma>1$; Corollary~5.23 in~\cite{villani2009optimal}, which shows that weak convergence $p_{t'}\to p_t$ implies $\nabla\phi_{t'}\to \nabla\phi_t$ in measure; and~\cite{Segers.2022}, which proves uniform convergence $\|\nabla\phi_{t'}-\nabla\phi_t\|_{L^\infty}\to 0$ under additional regularity assumptions.  
Quantitative stability estimates can be found in~\cite{manole2021plugin,delalande2021quantitative}.  
However, stability results establish only \emph{continuity} of the optimal transport map under perturbations and do not address finer \emph{differentiability} properties. 
In contrast, identifying the time derivative of the transport map is central to many of the aforementioned applications.

Overall, there are few results concerning the time regularity of optimal transport maps when the target measures $\{P_t\}_{t\in I}$ have time-varying supports $\{\Omega_t\}_{t\in I}$.

\subsubsection*{Results and Methodology}
In this work, we show that if the reference measure $Q$ is sufficiently regular in space, and the curve of probability measures $\{P_t\}_{t \in I}$ is sufficiently regular in space and evolves smoothly in time (in the sense of \cref{Definition:smoothcourve}), then the corresponding curve of optimal transport maps $\{\nabla \phi_t\}_{t \in I}$ also evolves smoothly in time. Moreover, its time derivative solves the \emph{linearized Monge--Ampère equation}, which is a second-order elliptic partial differential equation with strictly oblique boundary conditions and a null zero-order term (see \cref{Theorem:MainWithTime}).

Our approach follows the strategy of linearizing the Monge--Ampère equation (see \cite{figalli2017monge,Villani2003}) and applying the implicit function theorem. To handle the time-varying supports, we make essential use of the tool introduced in~\cite{urbas1997second}, which characterizes the evolution of the supports of the target measures via \emph{convex defining functions}.

The proof of our main result relies on analyzing the following functional derived from the Monge--Ampère equation:
\[
\Gamma(t,\phi)
=
\begin{pmatrix}
\Gamma_t^{(1)}(\phi)\\[4pt]
\Gamma_t^{(2)}(\phi)
\end{pmatrix}
=
\begin{pmatrix}
\log(\det(D^2 \phi)) - \log(q) + \log(p_t(\nabla \phi))\\[4pt]
h_t(\nabla \phi)
\end{pmatrix},
\]
where $h_t$ is a uniformly convex defining function for ${\rm supp}(P_t)$. 
Analyzing the invertibility of its Fr\'echet derivative leads to the study of the following second-order elliptic PDE:
\begin{align*}
    {\rm tr}\!\big([D^2 \phi_t]^{-1} D^2 \xi \big)
    + \frac{\langle \nabla p_t(\nabla \phi_t), \nabla \xi \rangle}{p_t(\nabla \phi_t)}
    &= f \quad \text{in } \Omega,\\
    \langle \nabla h_t(\nabla \phi_t), \nabla \xi \rangle
    &= g \quad \text{on } \partial \Omega.
\end{align*}
In \cref{thm: existence and uniqueness linearized MA}, we show that this PDE is solvable between the spaces
\[
\CXt
=
\left\{
\xi \in \CC^{2,\alpha}(\overline{\Omega}) : 
\int_{\partial \Omega_t} p_t \, \xi(\nabla \phi_t^*) = 0
\right\}
\]
and
\[
\CYt
=
\left\{
(f,g) \in \CC^{0,\alpha}(\overline{\Omega}) \times \CC^{1,\alpha}(\partial \Omega) :
\int_{\Omega} q f = \int_{\partial \Omega_t} p_t \, g(\nabla \phi_t^*)
\right\}.
\]
The proof then proceeds by applying the implicit function theorem to a re-centered version of $\Gamma$, since the range of $\Gamma$ generally exceeds the space $\CYt$.

Building on our theoretical results, we present two statistical applications. 
The first concerns the transport-based regression problem (\cref{sec:transport-based-regression}), where the conditional distribution $P_{Y \mid X = x}$ plays the role of the time-varying target measure. We show that the associated transport-based quantile regressor varies smoothly with respect to the covariate $x$.  
Our second application (\cref{sec:clt}) establishes a central limit theorem for the smoothed optimal transport map, whose limiting distribution is again characterized through the linearized Monge--Ampère equation.

\subsubsection*{Organization}

The remainder of the paper is organized as follows: \Cref{Sec:Notation} introduces and recalls key notations used throughout for the reader's convenience. In \Cref{sec:MainResult}, we state the main result and its assumptions. \Cref{sec:application} applies our theoretical results to various statistical problems.  \Cref{Section:Proof of main} is devoted to the proof of the main result.

\section{Notation}\label{Sec:Notation}

For any Borel probability measure $P$ over $\R^d$  (i.e., $P\in \CP(\R^d)$), let $\support(P)$ denote its topological support, which is defined to be $\support(P) := \R^d \setminus \bigcup\{U\subseteq \R^d: U\text{ is open }, P(U)=0\}.$
 The support of a function $f:\R^d \to \R$ is defined to be $\support(f):= \overline{\{x \in \R^d:f(x)\neq 0\}}$, where the overline the notes the topological closure. The interior of a set $A$ in a topological space $B$ is denoted as ${\rm int}(A)$.   We call a measure absolutely continuous if it is absolutely continuous with respect to the Lebesgue measure $\ell_d$. The same applies with densities, unless the contrary is stated, a density of a probability measure $P$ is the Radon–Nikodym derivative of $P$ w.r.t.\ $\ell_d$. The integration, of a function $f$ is always w.r.t. $\ell_d$ (or the corresponding Hausdorff measure $\mathcal{H}^{k}$ if a $k$-dimensional surface $S$ were involved)  and we simply write
 $ \int f=\int f {\rm d}\ell_d = \int f(x) {\rm d}x$ (or $\int_{S} f=\int_S f {\rm d}\mathcal{H}^k   $ for the surface case).  The indicator function of a Borel measurable set $A$ is denoted as ${\bf 1}_A$.

 Let $U\subseteq \R^d$ be bounded, $k\in \N, \alpha \in [0,1]$, we denote by $\CC^{k,\alpha}(U)$ the space consisting of functions whose $k$-th order partial derivatives are uniformly H\"older continuous with exponent $\alpha$. We adopt the convention that $\CC^{k,0}(U) = \CC^k(U)$. Let $f\in \CC^{1}(U)$, $e\in \R^d$, we denote by $\partial_{e} u$ the directional derivative along the direction $e$ and denote by $\nabla f := (\partial_{e_1}f,\dots, \partial_{e_d}f)^\T$, where $\{e_j\}_{j=1}^d$ denotes the canonical basis of $\R^d$. Hessian matrix of a function $f\in \CC^2(U)$ is denoted as $D^2f $. 
 %When $d=1$, we also write $f':=\nabla f$. 
 Let $F:\R^d \to \R^d$, we denote by $\diver{F}$ the divergence of $F$. Denote by $\R_+^d$ the half-space $\{x=(x_1,\dots, x_d)\in \R^d: x_d>0\}$.   We call {\it domain} every open, bounded, connected, and non-empty subset $\Omega$ of $\mathbb{R}^d$. Moreover, $\Omega$ is said to be a $\CC^{k,\alpha}$ domain $(k \in \N, \alpha \in [0,1])$ if for every $p \in \partial \Omega$, the boundary of $\Omega$, there exists a neighborhood $B=B(x_0)$ of $x_0$ in $\mathbb{R}^d$ and a diffeomorphism $\Psi : B \rightarrow D:=\Psi(B)\subset \R^d$ such that
 \begin{equation*}
     \label{eq:domain}
     \text{(i) $\Psi(B\cap \Omega)\subset \R_+^d$; \quad (ii) $\Psi(B\cap \partial \Omega) \subset\partial \R_+^d$;
\quad (iii) $\Psi\in \CC^{k,\alpha}(B), \Psi^{-1}\in \CC^{k,\alpha}(D)$}.
 \end{equation*}

As curve $\{b_t\}_{t\in I}$ in a Banach space $(B, \|\cdot\|_B)$ is $\CC^1$ if there exists a continuous function $\partial_t b_t:I\to B$ such that the limit 
\begin{equation}\label{eq:C1}
    \lim_{h\to0}\left\|\frac{1}{h}(b_{t+h} - b_{t}) -   \partial_t b_t\right\|_{B} = 0,\quad \text{for all}\,\, t\in I\,.
\end{equation} 
A mapping $F: \CU \to \CY$, were $(\CX, \|\cdot\|_\CX)$ and $(\CY, \|\cdot\|_\CY)$ are Banach spaces and $\CU\subset \CX$ is open,  is said to be Fr\'echet differentiable at $x\in \CU$ if there exists a bounded linear functional $A(x):\CX\to \CY$ such that 
\[
\lim_{h\to 0}\frac{\norm{F(x+h)-F(x)- A(x)h}_\CY}{\norm{h}_\CX} =0
\]
for every $h\in \CX$. We say that $F$ is Fr\'echet differentiable in $\CU$ if $F$ is Fr\'echet differentiable at every point $x\in\CU$. We say that $F$ is $\CC^1$ in $\CU$ if it is Fr\'echet differentiable in $\CU$ and the mapping $A: \CX \to \CL(\CX,\CY)$, where $\CL(\CX,\CY)$ denotes  the space of bounded linear functional from $\CX$ to $\CY$, is continuous w.r.t.\ the norm topology. We denote by ${\rm tr}(M)$ the trace of a matrix $M$. Given two symmetric matrices $A,B$ with dimension $d\times d$ over $\R$, we say that $A\leq B$ if $B-A$ is positive semi-definite. For quantities $a,b$, we write $a \lesssim b$ if there exists a constant $C > 0$ (which may depend on other parameters depending on context) such that $a \leq C b$.

\section{Main Result}\label{sec:MainResult}
In this section, we state the main result and introduce its assumptions. We consider a curve $\{P_t\}_{t\in I}$ of probability measures with both smooth (and uniformly convex) supports and densities. 

As mentioned in the introduction, our main result accommodates time-dependent supports. Due to this dependency, the curve $\{p_t\}_{t\in I}$ is not well-defined in $\CC^{1, \gamma}(\Omega')$ as in the case $P_t = {\rm Unif}[0,1+t]$ for $t\in (0,1)$. To address this issue, we will start assume that $\{p_t\}_{t \in I}$ has a $\CC^{1,\gamma}$ extension.  Later, in \cref{Lemma:extension}, we will provide sufficient conditions for this extension.

To begin with, we define our notion of a \(\CC^1\) curve of probability measures in \(\CC^{1,\alpha}(\overline{\Omega'})\) for a given set \(\Omega'\). {The definition below captures two aspects of the curve of probability measures: changes in support and changes in the density functions.}
\begin{definition}\label{Definition:smoothcourve}
    A curve $\{P_t\}_{t\in I}$ of probability measures over a domain $\Omega'$ is said to be $\CC^1$ in $\CC^{1, \alpha}(\overline{\Omega'})$ if the following conditions are satisfied:
    \begin{enumerate}[label = (\arabic*)]
        \item There exists $\kappa\in (0, +\infty)$ and a $\CC^1$ curve of convex functions $\{h_t\}_{t\in I}$ in $\CC^{2,\alpha}(\overline{\Omega'})$ such that   $$\Omega_t:={\rm int}(\support (P_t)) = \left\{ y \in \mathbb{R}^d : h_t(y) < 0 \right\},$$
        with $ \|\nabla h_t\|=1 $ on $\partial \Omega_t$, 
 $\bigcup_{t\in I}\Omega_t+\frac{1}{\kappa}\mathbb{B}:=\left\{y: {\rm dist}\left(y, \bigcup_{t\in I}\Omega_t \right)<\frac{1}{\kappa}\right\}\subset \Omega'$,
        and 
        $$ \frac{1}{\kappa} \Id \leq D^2 h_t \leq \kappa \Id \quad \text{in } \Omega'.  $$
        \item There exists a $\CC^1$ curve  $\{\log(p_t)\}_{t\in I}$ in $\CC^{1,\alpha}(\overline{\Omega'})$ such that $P_t=p_t {\bf 1}_{\Omega_t}$.  
    \end{enumerate}
\end{definition}

The function $h_t$ is commonly known as the \emph{convex defining function} of $\support(P_t)$. It is noteworthy that $ \Omega_t =\left\{ y \in \mathbb{R}^d : h_P(y) < 0 \right\}$ represents an open, bounded, and uniformly $\CC^{2,\alpha}$ convex domain. Conversely, given a uniformly convex set $\Omega_t$, one can construct the function $h_t$ to be smooth and uniformly convex, approximating $- \text{dist}(\cdot, \partial \Omega_t) + \frac{1}{2} \text{dist}(\cdot, \partial \Omega_t)^2$ near $\partial \Omega_t$. For example, if $\Omega_t = \mathbb{B}(0, 1)=\{x\in \R^d: \|x\|=1\}$, one might choose $h_P(y) := \frac{1}{2} (\|y\|^2 - 1)$ (see, e.g., \cite[pp~40]{figalli2017monge}). 

\begin{theorem}\label{Theorem:MainWithTime}
Let $\Omega$ be a $\CC^2$ uniformly convex domain and let $\Omega'\subset \R^d$ be a domain. If the curve $\{P_t\}_{t\in I}$ is $\CC^1$ in $\CC^{1, \gamma}(\overline{\Omega'})$ with $0<\gamma\leq 1$ (in the sense of \cref{Definition:smoothcourve}) and let $\log(q)\in \CC^{\alpha}(\overline{\Omega})$ with $0< \alpha<\gamma$, then the curve $ \{\nabla \phi_t\}_{t\in I}\subset  \CC^{1,\alpha}(\overline{\Omega})$, where $\phi_t$ solves
\eqref{eq:MongeAmpere}, is also $\CC^1$. Moreover, the time derivative of $I \ni t\mapsto \phi_t \in \CC^{2,\alpha}(\overline{\Omega})$, denoted as $\partial_t \phi_t$, is the unique solution (up to an additive shift) of the linearized Monge--Amp\`ere equation 
\begin{equation}\label{LinearizedMA}
\begin{aligned}
       {\rm tr}([D^2\phi_{t}]^{-1} D^2\xi ) +\frac{ \langle \nabla p_{t}(\nabla \phi_{t}), \nabla \xi\rangle  }{p_{t}(\nabla \phi_{t})} & = -\frac{\partial_t p_t(\nabla \phi_t)}{p_t(\nabla \phi_t)}   \quad \rm in\,\, \Omega,\\
  \langle \nabla h_t(\nabla \phi_{t}), \nabla \xi\rangle&=-\partial_t h_t(\nabla \phi_t) \quad \rm on \,\, \partial \Omega.
\end{aligned}
\end{equation}
\end{theorem}
\begin{remark}
\begin{enumerate}[label=(\roman*)]
    \item 
    As we will see in \cref{lemma: holder composition stability}, 
requiring the curve $\{P_t\}_{t\in I}$ to be $\CC^1$ in $\CC^{1,\gamma}(\overline{\Omega'})$ (rather than merely in $\CC^{1,\alpha}$) is necessary to ensure the continuity of the map 
$\phi \mapsto \nabla p_t \circ \phi$.  
This level of regularity is also assumed in \cite{urbas1997second}.
    \item
    Recent advances on the Monge--Ampère equation \cite{chen2021global} 
    suggest that the assumptions on $\Omega$ and $\{\Omega_t\}_{t\in I}$ 
    could be relaxed to convexity and $\CC^{1,1}$ boundary regularity.  
    However, to avoid additional technical complications in our arguments, 
    we adopt the classical assumptions used in \cite{caffarelli1996boundary,urbas1997second}, 
    which are sufficient to cover the primary applications we have in~\cref{sec:application}.

    \item
    The positivity of $\alpha$ is crucial for the application of 
    Schauder theory in \cite{gilbarg1983trudinger}, 
    and therefore cannot be relaxed using our current techniques.
\end{enumerate}
\end{remark}

In \cref{Theorem:MainWithTime}, we implicitly assume that each density $p_t$, originally defined only on $\Omega_t$, admits an extension to the fixed domain $\Omega'$.  
The following result provides a sufficient condition ensuring the existence of such an extension, thereby justifying the applicability of \cref{Theorem:MainWithTime}. We present its proof below, while the proof of \cref{Theorem:MainWithTime} is deferred to~\cref{Section:Proof of main}.

\begin{lemma}\label{Lemma:extension}
    Assume that $I$ is open. Let $\{h_t, \Omega_t\}_{t\in I}$ and $\Omega'$ be defined as in \cref{Definition:smoothcourve}, let $h:(t,x) \mapsto h(t,x)= h_t(x)$ be $\CC^{1,\alpha}$ over $I\times \Omega'$, and let $ \log(p): (t, x) \mapsto  \log(p(t, x))=\log(p_t(x))$ be a  $\CC^{1,\alpha}$ function over the set 
    $ \mathcal{D} :=\{ (t, x): \ t\in I,  \ x \in \overline{\Omega}_t \}.$
    Then the curve $\{P_t\}_{t\in I}$ with $P_t=p_t {\bf 1}_{\Omega_t}$ is $\CC^1$ in $\CC^{1, \alpha}(\overline{\Omega'})$ in the sense of \cref{Definition:smoothcourve}.
\end{lemma}
\begin{proof}
    We want to extend $\log p$ from ${\rm int}(\CD):={\rm int}\left( \{(t, x): \,t\in I,  \ x \in \Omega_t\}\right)$ to $I\times \Omega'$ using \cite[Lemma~6.37]{gilbarg1983trudinger}. Since $I$ is open and connected, we may assume without loss of generality that $I= \R$. Then, for every $ (t_0,x_0)\in \partial \CD$, it suffices to find a neighborhood $B(t_0,x_0)$ of $(t_0,x_0)$ and a $\CC^{1,\alpha}$ diffeomorphism $\Psi: B(t_0,x_0)\to D:= \Psi(B(t_0,x_0))\subset \R^{d+1}$ such that $\Psi(B(t_0,x_0)\cap \CD) \subset \R^{d+1}_+$, $\Psi(B(t_0,x_0)\cap \partial \CD) \subset \partial \R^{d+1}_+$, and $\Psi\in \CC^{1,\alpha}(B(t_0,x_0)), \Psi^{-1}\in \CC^{1,\alpha}(D)$. A candidate map is
    \[
    \Psi: \R^{d+1}\ni (t,x_1,\dots,x_{d-1},x_d)\mapsto (t, x_1,\dots, x_{d-1}, -h_t(x)) \in \R^{d+1}.
    \]
   Since $h_t$ is a convex defining function of $\Omega_t$, the first two requirements are easily satisfied, and $\Psi$ is $\CC^{1,\alpha}$. We now show that $\Psi$ is one-to-one in a small neighbourhood of $(t_0,x_0)$. As $\{x\in \Omega: h_t(x)<0\}$ is open and $h_t$ is uniformly convex, one may choose $\delta$ small enough such that $\|\nabla h_t(x)\| \neq 0$ for all $(t,x)\in B:= \mathbb{B}((t_0,x_0),\delta)$. We proceed with the proof by contradiction. Suppose there exists $(t,x) =(t,x_1,\dots, x_d)\neq (s, y)=(s,y_1,\dots,y_d)\in B$ such that $\Psi(t,x)= \Psi(s, y)$, then 
    \[
    t=s,\,\, x_i=y_i,\quad i=1,\dots, d-1,
    \]
    and $h_t(x)= h_t( y)$. Define $\tilde h: z \mapsto h_t(x_z)$, where $x_z:= (x_1,\dots,x_{d-1},z)$, then $\tilde h$ is also $\CC^{1,\alpha}$ uniformly convex. Moreover, as $|\tilde h'(z)| \neq 0$ over $\CZ:=\{z: (t,x_z)\in B\}$, $\tilde h'(z) = \inner{\nabla h_t(x_z)}{e_d}$ is either positive or negative over $\CZ$ by the monotonicity of the gradient of a uniformly convex function.
    Therefore, the fundamental theorem of calculus yields that
    \[
    0=\tilde h(y_d) - \tilde h(x_d) = (y_d-x_d)\int_0^1 \tilde h'(x_d + t(y_d-x_d))\,dt \neq 0,
    \]
    which leads to a contradiction.
\end{proof}

Under the assumption that the supports $\{\Omega_t\}_{t \in I}$ remain constant in $t$, we no longer require that $\{P_t\}_{t \in I}$ forms a $\CC^1$ curve in the sense of \cref{Definition:smoothcourve}; it suffices to be a $\CC^1$ curve in the standard sense. In this case, we recover the classical result of \cite{loeper2005regularity}. It is worth noting that the following theorem is already sufficient to derive asymptotic confidence intervals for the optimal transport map, as in \cite{manole2023central}.
\begin{corollary}\label{Coro:IntroSimple}
{Let $\Omega$ (resp. $\Omega'$) be a $\CC^2$ (resp. $\CC^{2,\gamma}$ with $0<\gamma\leq 1$) uniformly convex domain.} Let $\log(p_t)\in \CC^{1, \gamma}(\overline{\Omega'})$ and let $\log(q)\in \CC^{\alpha}(\overline{\Omega})$ with $0< \alpha<\gamma$ such that
$ 1=\int_{\Omega}q=\int_{\Omega'} p_t=1$ for all $t\in I$.   If the curve $\{\log(p_t)\}_{t\in I}\subset \CC^{1,\gamma}(\overline{\Omega'})$ is $\CC^1$ (in the sense of~\eqref{eq:C1}), then the curve $ \{\nabla \phi_t\}_{t\in I}\subset  \CC^{1,\alpha}(\overline{\Omega})$, where $\phi_t$ solves
\eqref{eq:MongeAmpere} for $\Omega_t=\Omega'$, is also $\CC^1$. Moreover,  the time derivative of $I \ni t\mapsto \phi_t \in \CC^{2,\alpha}(\overline{\Omega})$, denoted as $\partial_t \phi_t$, is the unique solution (up to an additive shift) of the linearized Monge--Amp\`ere equation 
\begin{equation*}
    \begin{aligned}
           {\rm tr}([D^2\phi_{t}]^{-1} D^2\xi ) +\frac{ \langle \nabla p_{t}(\nabla \phi_{t}), \nabla \xi\rangle  }{p_{t}(\nabla \phi_{t})} & = -\frac{\partial_t p_t(\nabla \phi_t)}{p_t(\nabla \phi_t)}\quad \rm in\,\, \Omega\,,\\
  \langle\nu_{\Omega'}(\nabla \phi_{t}), \nabla \xi\rangle&=0 \quad \rm on \,\, \partial \Omega\,,
    \end{aligned}
\end{equation*}
where $\nu_{\Omega'}$ denotes the unit outer normal vector-field to $\partial \Omega'$. 
\end{corollary}

\Cref{Coro:IntroSimple} follows immediately from \cref{Theorem:MainWithTime} once we show that the family $\{\log(p_t)\}_{t \in I} \subset \CC^{1,\gamma}(\overline{\Omega'})$ can be extended to a larger set $\Omega'_\kappa := \Omega' + \frac{1}{\kappa}\mathbb{B}$ for some $\kappa$ depending on the uniform convexity of $\Omega'$ (see \cref{lemma:invariant-support} below).  This extension is crucial for computing the Fr\'echet derivative of the target functional (see~\cref{LemmaFrechet}). In contrast with \cref{Lemma:extension}, the fact that the support does not depend on $t$ eliminates the main challenge to our analysis. As a consequence, we no longer need to impose any extra regularity on the curve $t\mapsto \log(p_t)$.
\begin{lemma}\label{lemma:invariant-support}
    Under the assumptions of \cref{Coro:IntroSimple}, the curve $\{P_t\}_{t\in I}$, with $P_t=p_t {\bf 1}_{\Omega'}$, is $\CC^1$ in $\CC^{1, \gamma}(\overline{\Omega'})$ (in the sense of \cref{Definition:smoothcourve}).
\end{lemma}
\begin{proof}
{Fix $t\in I$. Since $\Omega' $ is $\CC^{2,\gamma}$, for each  $x\in \partial \Omega'$, there exists a ball $B(x)=\mathbb{B}(x, \epsilon_x)$ and a $\CC^{2,\gamma}$ diffeomorphism $\Psi=\Psi_{
x
 }: B(x) \to D_x:=\Psi_x(B(x))\subset \R^d$ that straightens the boundary (see \eqref{eq:domain}).} Since  $\partial \Omega'\subset  \bigcup_{x\in \partial \Omega' }\mathbb{B}\left(x, \frac{\epsilon_x}{2}\right)$ and $\partial \Omega'$ is compact, there exists a finite covering $\partial \Omega' \subset \bigcup_{i=1 }^n \mathbb{B}\left(x_i, \frac{\epsilon_{x_i}}{2}\right)$,  with $x_i\in \partial \Omega'$. Let $\kappa>0$ be such that $\Omega'\subset \Omega_{\kappa}'\subset \bigcup_{i=1 }^n \mathbb{B}\left(x_i, \frac{\epsilon_{x_i}}{2}\right)\cup \Omega'  $.

 Denote by $u_t:= \log p_t$ for ease of notation. Set $i\in \{1, \dots, n\}$ and  $$D_+^{(i)}:=\Psi(\Omega'\cap B(x_i))= D_{x_i}\cap \R_+^d.$$ Setting $\tilde u_t^{(i)}(y):= u\circ \Psi^{-1}_{x_i}(y)$ for every $y=(y_1,\dots, y_d) =: (y',y_d)\in D_+^{(i)}$, the extension of $\tilde u_t$ to $ D_{x_i}$ is given by
  \[
  \tilde w_t^{(i)} (y',y_d) := -3 \tilde u_t^{(i)}(y',-y_d) + 4\tilde u_t^{(i)}(y',-y_d/2), \quad y_d<0, 
  \]
  so that   $\tilde w_t^{(i)} \in \CC^{1,\gamma}(D_{x_i})$ (see \cite[Lemma~6.37]{gilbarg1983trudinger}) and 
  $$ w_t^{(i)}:=\begin{cases}
      \tilde w_t^{(i)}\circ \Psi_{x_i} &{\rm if}\ x\notin \Omega' \\
       u_t &{\rm if}\ x\in \Omega'\\
  \end{cases}  $$
is a $\CC^{1, \gamma}$ extension of $u_t$ in $ B(x_i)\cup \Omega'$.   Let $\{\eta_i\}_{i=1}^n$ be a locally finite partition of unity subordinate to the covering $\{B(x_i)\}_{i=1}^n$.  Then we have 
$$ w_t=\sum_{i=1}^n \eta_i w_t^{(i)}\in \CC^{1, \gamma}(\overline{\Omega_{\kappa}'}) $$
and $w_t=u_t$ in $\overline{\Omega'}$
for all $t\in I$. Since $u_t = \log(p_t)$ is $\CC^1$ in $\CC^{1,\gamma}(\overline{\Omega'})$ in the standard sense, \cref{lemma: holder composition stability} (below) shows that for each $i$, the curve $\{w_t^{(i)}\}_{t\in I}$ is  $\CC^{1}$ in $\CC^{1,\gamma}\left(\overline{\mathbb{B}\left(x_i, \epsilon_{x_i}/2\right)\cup \Omega'}\right)$, which {\it a fortiori
} implies that $\{w_t\}_{t\in I}$ 
 is $\CC^1$ in $\CC^{1,\gamma}(\overline{\Omega'_\kappa})$. This concludes the proof. 
\end{proof}

\section{Applications}\label{sec:application}
\subsection{Transport-based Quantile Regression}\label{sec:transport-based-regression}
Our first application concerns the \emph{transport-based quantile regression} problem~\cite{carlier2016vector,delBarrioQuantile2024}. Let $(X, Y) \in \mathbb{R}^{m+d}$ be two random vectors with joint probability distribution $P_{X,Y}$. The transport-based quantile regressor $(u,x) \mapsto \nabla_u \phi(u,x)=\mathbb{Q}_{Y\mid X}(u,x)$ of random vectors $Y$ on $X$ is defined as the unique Borel function such that, for $P_X$-almost every $x$, the function $u \mapsto \nabla_u \phi(u,x)$ is the transport-based quantile function of $P_{Y \vert X=x}$---the conditional probability measure of $Y$ given $X=x$ ---which we assume to have  density $p_{Y \vert X}(\cdot,x)$~\cite{Chernozhukov2017, HallinBarrio2021, ghosal2022multivariate,Shi2022}. Recall that transport-based quantile function of a probability measure $P$ as the unique gradient of a convex function pushing a fixed reference measure $Q$ (with density $q$ w.r.t.\ the Lebesgue measure) forward to $P$. Therefore, for each $x \in {\rm int}({\rm supp}(P_X))$, $u \mapsto \phi(u,x)$ solves the \textit{conditional Monge--Ampère equation}~\cite{carlier2016vector}:
\begin{align*}
    \det(D^2_u \phi(\cdot, x)) &= \frac{q}{ p_{Y \vert X}(\nabla_u \phi(\cdot,x),x)} \quad {\rm in} \ \Omega, \\
     \nabla \phi (\Omega \times \{x\}) &= {\rm supp} (P_{Y \vert X}(\cdot, x)). 
\end{align*}
In this setting, our goal translates into studying the regularity of the transport-based quantile regressor $(u,x) \mapsto \nabla_u \varphi(u,x)$ with respect to the covariates $x$.

\begin{figure}[t]
    \centering
    \includegraphics[width=0.5\linewidth, trim={1cm 2cm 1cm 2.5cm},clip]{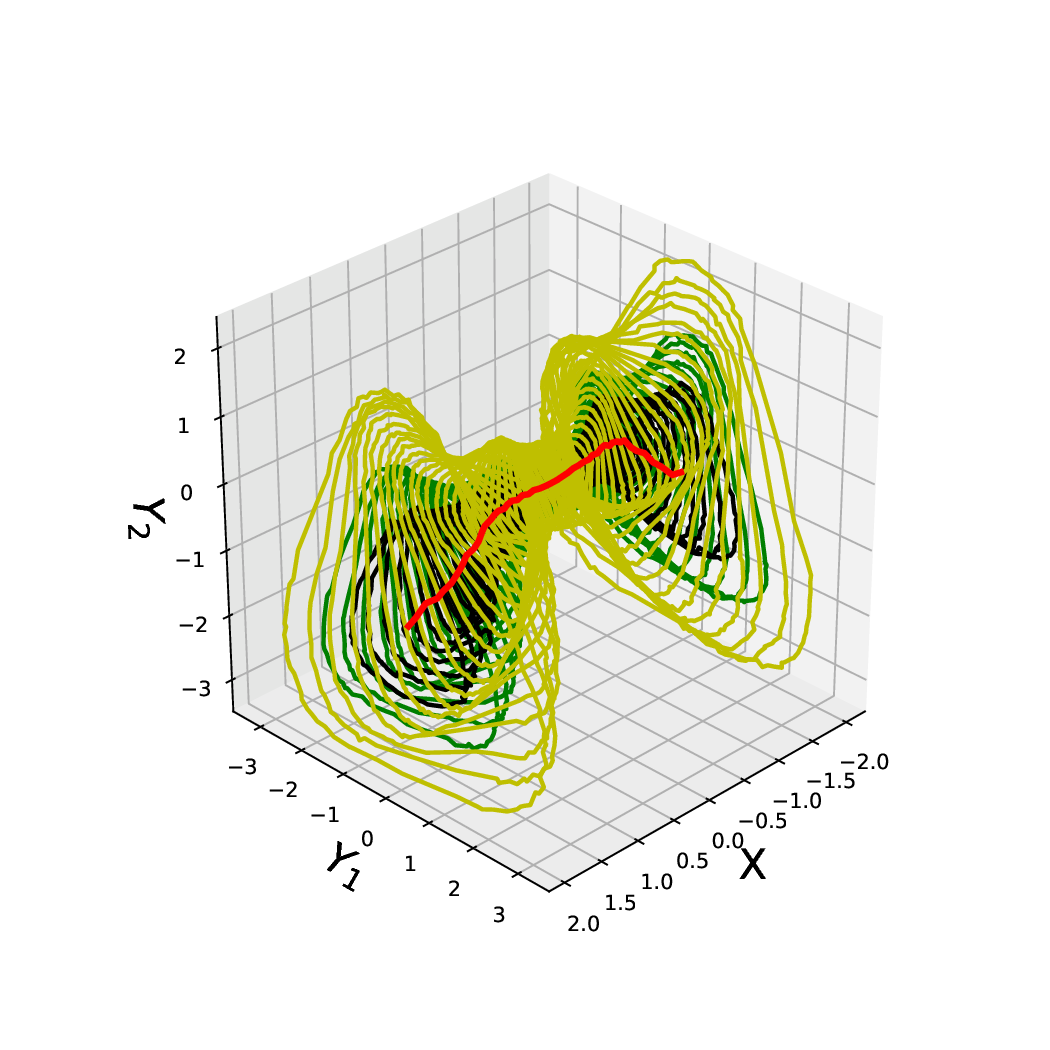}
    \caption{Multivariate quantile regression (two-dimensional variable of interest $\BY=(Y_1,Y_2)^\T$; univariate regressor $X$), showing the conditional medians (red) and the conditional quantile contours of order $\tau = 0.2$ (black), $\tau = 0.4$ (green), $\tau = 0.8$ (yellow). The number of samples generated is $n = 320,050$. The nonparametric estimator is that of~\cite{delBarrioQuantile2024}, trivially adapted to the change of reference measure.}
    \label{fig:quantile}
\end{figure}

\Cref{fig:quantile} illustrates the expressive capabilities of the transport-based quantile regressor in visualizing complex data. It effectively captures nonlinear trends, heteroskedasticity, and the overall shape of the conditional distributions. In the case $m=1$ and $d=2$, we visualize the regression median (red) together with quantile tubes of orders $\tau = 0.2$ (black), $0.4$ (green), and $0.8$ (yellow), defined via the \emph{conditional quantile contours}, i.e., the image of the mapping
$x \mapsto \bigl(x, \mathbb{Q}_{Y \mid X}\bigl(\partial \mathbb{B}(0,\sqrt{\tau}), x\bigr)\bigr)$, 
for the model
\[
\mathbf{Y} =
\begin{pmatrix}
Y_1 \\
Y_2
\end{pmatrix}
=
R(X)
\begin{pmatrix}
\frac{1}{2}(X^2+1)Z_1 \\
-\frac{1}{4}X^2 + \frac{1}{2}(X^2+1)Z_2
\end{pmatrix},
\]
where $X \sim {\rm Unif}[-2,2]$ and
\[
R(X) =
\begin{pmatrix}
\cos\left(\frac{\pi}{2}(2-X)\right) & -\sin\left(\frac{\pi}{2}(2-X)\right) \\
\sin\left(\frac{\pi}{2}(2-X)\right) & \cos\left(\frac{\pi}{2}(2-X)\right)
\end{pmatrix}
\in \R^{2\times 2},
\qquad
\BZ := (Z_1,Z_2)^\top = \tilde{\BZ}\,{\bf 1}_{\{\|\tilde{\BZ}\|\le 100\}}.
\]
The random vector $\tilde{\BZ}$ follows a four-component Gaussian mixture, where each component has equal weight, common covariance matrix $\Sigma = 10\,\Id$, and means $(0.866,-0.5)^\top$, $(-0.866,-0.5)^\top$, $(0,0)^\top$, and $(0,1)^\top$, respectively.  
The reference measure for the quantile regressor in Figure~\ref{fig:quantile} is the uniform distribution on the unit ball.

From \cref{Theorem:MainWithTime} and \cref{Lemma:extension}, we obtain as a corollary the desired regularity of the transport-based quantile regressor, which in particular apply to the toy example above. This smoothness can in turn be used to develop asymptotic theory~\cite{KoenkerBook} and to perform causal and counterfactual inference~(see, e.g., \cite{ImbensNewey2009, Chernozhukov2013}), as in the classical one-dimensional case.

\begin{corollary}
Let $\Omega$, $\Omega'$  and the reference measure $Q={\bf 1}_{\Omega} q dx$   be as in \cref{Theorem:MainWithTime}. 
Let $(X, Y) \in \mathbb{R}^{m+d}$ be a pair of random variables with probability law $P_{X, Y}=P_{Y\vert X} P_X$. Assume that 
\begin{enumerate}[label = (\arabic*)]
      \item for each  $x\in {\rm int}({\rm supp}(P_X))$, the conditional  probability measure has support within $\Omega'$ and the log density $(x,y) \mapsto \log p_{Y\vert X}(y, x)$  is $\CC^{1,\gamma}$ over $\overline{{\rm supp}(P_{X,Y})}$ and 
    \item the convex defining function $h(\cdot, x)$ of ${\rm supp}(P_{Y\mid X}(\cdot, x)))$  is $\CC^{1,\gamma}$ over $\Omega' \times {\rm supp}(P_X)$.
\end{enumerate}
Then  the function $ {\rm int}({\rm supp}(P_x)) \ni x \mapsto  \mathbb{Q}_{Y \vert X}(\cdot , x) \in \CC^{1,\alpha}(\overline{\Omega})$ is $\CC^1$. In particular,  $\nabla_x \mathbb{Q}_{Y \vert X}\in \CC(\overline{\Omega} \times {\rm int}({\rm supp}(P_X)))$. 
\end{corollary}

\subsection{Central Limit Theorem for Smooth Optimal Transport Map}\label{sec:clt}
The smooth optimal transport problem is introduced as a way to alleviate the curse of dimensionality in estimating the Wasserstein distance~\cite{Goldfeld2020, goldfeld2020gaussian, goldfeld2020asymptotic, nietert2021smooth, GoldfeldKato2024}. 
Let $\Omega$ be a $\CC^{2,1}$ uniformly convex domain. 
Let $K : \overline{\Omega}\times\overline{\Omega} \to (0,\infty)$ be a $\CC^{2,1}$ kernel satisfying 
$\int_{\Omega} K(x,y)\,{\rm d}y = 1$ and $\int_{\Omega} K(x,y)\,{\rm d}x = 1$ for all $(x,y)\in \Omega\times\Omega$.  
Let the reference measure $Q \in \CP(\Omega)$ satisfy $\log(q) \in \CC^{\alpha}(\overline{\Omega})$ for some $0<\alpha<1$.  
The smooth (or convoluted) optimal transport problem between $P \in \CP(\Omega)$ and $Q$ is defined as
\begin{equation}
\label{eq:Smooth-OT}
    \min_{\pi \in \Pi(P_K, Q)} \int \|x-y\|^2\,d\pi(x,y),
\end{equation}
where $P_K$ admits the smoothed density $\int K(\cdot,x)\,d P(x)$. 

Suppose we observe i.i.d.\ samples $X_1, \dots, X_n$ from $P$ and define the empirical measure  
$\widehat{P}_n := \frac{1}{n} \sum_{i=1}^n \delta_{X_i}$.  
Let $T = \nabla\phi$ be the optimal transport map from $Q$ to $P_K$, and let $\widehat{T}_n$ denote the optimal transport map from $Q$ to $\widehat{P}_{n,K}$.  
In this section, we establish a central limit theorem for the smooth optimal transport maps $\widehat{T}_n$ (see~\cref{prop:clt}).

As expected, the result follows from the functional delta method.  
To apply it, one needs the Hadamard differentiability of the mapping $\Phi$ that sends a density $p$ to the optimal transport map $\nabla\phi_{q\to p}$ pushing $q$ forward to $p$.  
This differentiability follows from our general result \cref{Coro:IntroSimple}, together with the lemma below.  
We remark that this lemma provides a general principle showing that smoothness along every $\CC^1$ curve implies Hadamard differentiability, which may be of independent interest to the reader.

\begin{lemma}\label{lemma:curve-hadamard}
   Let $(\CX,\|\cdot\|_\CX)$, $(\CY, \|\cdot\|_\CY)$ be normed spaces. Assume that $\mathcal{K}\subset \CX$ is open.   Let $\Phi: \mathcal{K}\to \CY$ be a functional. Suppose that for any $\CC^1$ curve $\eta:[0,1] \mapsto \mathcal{K}$ the function $\Phi\circ \eta:[0,1] \to \CY$ is also $\CC^1$. Then $\Phi$ is Hadamard differentiable at any $x_0\in \mathcal{K}$. That is, for any  $t_k\to 0^+$ and $h_k\to h\in \CX$ with $\{h_k\}_{k\geq 0}\subset \CX$, it follows that  
   $$ \lim_{k\to \infty}\frac{\Phi(x_0+ t_k h_k)-\Phi(x_0)}{t_k} =  D\Phi(x_0)(h)=\frac{d}{dt}\bigg\vert_{t=0} \Phi(x_0+ h t ). $$
\end{lemma}
\begin{proof}
    For any $h \in \mathcal{X}$, the curve $x_t = x_0 + th $ is $\CC^1$, for $t$ small enough, then there exists a mapping $D\Phi(x_0):\mathcal{X}\to \CY$ such that
    \[
    \lim_{t\to 0^+}\left\|\frac{\Phi(x_{t}) - \Phi(x_0)}{t} - D\Phi(x_0)(h)\right\|_{\CY} = 0\,.
    \]
    That is, $\Phi$ is Gateaux differentiable at $x_0$. Note that $\Phi$ is positive homogeneous as for $\lambda>0$, 
    \begin{equation*}\label{eq:Homogeneous}
        D\Phi(x)(\lambda h)=    \lim_{t\to 0^+}\frac{\Phi(x_0 + t\lambda h) - \Phi(x_0)}{t}=     \lim_{s\to 0^+} \frac{\Phi(x_0 + s h) - \Phi(x_0)}{s/\lambda} =      \lambda D\Phi(x)(h)\,.
    \end{equation*}
We proceed by contradiction. Suppose that 
    \[
    \lim_{k}\left\|\frac{\Phi(x_0 + t_k h_k) - \Phi(x_0)}{t_k} - D\Phi(x_0)(h)\right\|_{\CY}=\alpha\in (0,\infty]\,. 
    \]
    for some sequence $t_k\to 0^+$ and $h_k\to h\in \CX$ with the convention that $\alpha=\infty$ if the sequence $\left\{\|(\Phi(x_0 + t_k h_k) - \Phi(x_0))/t_k\|_{\CY}\right\}_{k\geq 1}$ diverges. The above condition shows that 
    \begin{equation}
        \label{eq:to-be-contracitied-Hadamard}
        \lim_{k}\left\|\frac{\Phi(x_0 + t_k h_k) - \Phi(x_0 + t_k h)}{t_k}\right\|_{\CY}>0\,.
    \end{equation}
    Define $\eta_k: s\mapsto x_0 + t_k (h + s(h_k-h) ) \in \CK$ for $k$ large enough, then $\Phi\circ \eta_k$ is $\CC^1$ in $\CY$. Hence, the mean-value theorem (see e.g.,~\cite[Theorem~4.2.]{Lang.functional.93}) and the positive homogeneity of $D\Phi$ give that 
    \begin{equation}\label{eq:mean-value}
    \begin{aligned}
                 \left\|\frac{\Phi(x_0 + t_k h_k) - \Phi(x_0 + t_k h)}{t_k} \right\|_\CY &\leq  \sup_{s\in [0,1]}\|D\Phi(\eta_k(s)) \left(h_k-h\right)\|_\CY\\
     & = \|h_k-h\|_\CX \sup_{s\in [0,1]}\left\|D\Phi(\eta_k(s)) \left(\frac{h_k-h}{\|h_k-h\|_\CX}\right)\right\|_\CY\,.
    \end{aligned}
    \end{equation}
 Write $v_k=\frac{h_k-h}{\|h_k-h\|_\CX}$. If $\{\|D\Phi(\eta_k(s))(v_k)\|_{\CY}: k\in \N,\, s\in [0,1]\}$ is bounded, then~\eqref{eq:mean-value} yields that
$$ \left\| \frac{\Phi(x_0 + t_k h_k) - \Phi(x_0 + t_k h)}{t_k}\right\|_\CX \lesssim \|h_k-h\|_\CX\to 0, $$
contradicting \eqref{eq:to-be-contracitied-Hadamard}. 
Therefore, as $h_k\to h$ and $t_k\to 0^+$, after taking subsequences, we can choose a sequence $\{x_n\}_{n\geq 1}\subset\CX$ such that $\|x_n-x_0\|_{\CX}\leq 2^{-n}$  and $\|D\Phi(x_n)(v_n)\|_{\CY}\geq 2^{n^2}$.  Applying \cite[Corollary~2.10]{Kriegl1997} shows that there exist a $\mathcal{C}^\infty$ curve $\eta:[-1,1]\to \CX$ and two sequences of positive numbers $\{s_n\}_{n\geq 1}\subset (0,\infty)$ and $\{s_n'\}_{n\geq 1}\subset (0,\infty)$ with $s_n,s_n'\to 0$ and such that $\eta(s_n +t)= x_n + t  \frac{v_n}{2^n}$  for $t\in (-s_n',s_n')$. Since $\Phi\circ \eta$ is $\CC^1$ by assumption, we derive that
\begin{equation}\label{eq:derivative}
    \frac{d}{dt}\bigg|_{t=0} \Phi(\eta(s_n +t))= \frac{1}{2^n} D\Phi(x_n)(v_n). 
\end{equation}
Moreover, $\left\|\frac{d}{dt}\bigg|_{t=0} \Phi(\eta(s_n +t))\right\|_{\CY}$ is finite for all $n\in \N$. On the other hand, \eqref{eq:derivative} and the fact that $\|D\Phi(x_n)(v_n)\|_{\CY}\geq 2^{n^2}$ imply
\[
\left\| \frac{d}{dt}\bigg|_{t=0} \Phi(\eta( s_n +t))\right\|_\CY= \frac{1}{2^n} \left\|D\Phi(x_n)(v_n)\right\|_\CY > 2^{n^2-n} \to \infty,
\]
which leads to a contradiction.
\end{proof}

As a consequence, we arrive at the following result.
\begin{corollary}\label{coro:hadamard-otmap}
The functional $ \Phi:\CC^{2}(\overline{\Omega}) \to\CC^{1,\alpha}(\overline{\Omega})$, mapping a density $p\in \CC^{2}(\overline{\Omega})$ to its transport map  $\nabla \phi_{q\to p}$ from $q$ to $p$, is Hadamard differentiable at $P_K$ tangentially to   $$\CX=\left\{\text{$f\in \CC^{2}(\overline{\Omega})$ with $\int_\Omega f=0$}\right\}. $$
\end{corollary}
\begin{proof}
 Since $P_K$ is uniformly bounded away from zero,  there exists an open neighborhood $\mathcal{U}\subset \CX$ of zero such that $P+f$ is a probability density. The function $\mathcal{U}\ni f\mapsto \Phi(P_K+f)$ is differentiable along curves in the sense of \cref{lemma:curve-hadamard}. Hence, applying \cref{Coro:IntroSimple} concludes the result.
\end{proof}

 Now, we ready to show the central limit theorem. Recall that for any Banach space $(\CX,\|\cdot\|_\CX)$, we say that a sequence of $\CX$-valued random elements  converges weakly to $X$, write as $X_n \overset{d}{\to} X$, if for any bounded continuous function $f:\CX \to \R$, $\E[f(X_n)] \to \E[f(X)]$.
\begin{proposition}\label{prop:clt}
          It follows that
 \begin{equation}\label{CLT-smooth-OT}
      \sqrt{n}\left( \widehat{T}_n - T \right)\xrightarrow{d}  \nabla \tilde{\mathbb{G}}
 \end{equation}
in $\CC^{1,\alpha}(\overline{\Omega})$, where $ \tilde{\mathbb{G}}$ solves the following equation a.e.,
\begin{align*}
       {\rm tr}([D^2\phi]^{-1} D^2\xi ) +\frac{ \langle \nabla p_{K}(T), \nabla \xi\rangle  }{p_{K}(T)} & = -\frac{\tilde{ \mathbb{G}}}{p_K(T)}   \quad \rm in\,\, \Omega\,,\\
  \langle \nabla h_t(T), \nabla \xi\rangle&=0 \quad \rm on \,\, \partial \Omega\,
\end{align*}
and $\tilde{\mathbb{G}}$ is the centered tight Gaussian process in $\CC^{2}(\overline{\Omega})$ with covariance function 
 $$ \E[\tilde{\mathbb{G}}(x) \tilde{\mathbb{G}}(y)] = {\rm Cov}\left( K(y, X_1)K(x, X_1)\right)\,.$$ 
\end{proposition}
 \begin{proof}
     In view of  $\widehat{P}_{n,K} = \frac{1}{n}\sum_{i=1}^n K(\cdot, X_i) \in \CC^{2,1}(\overline{\Omega})$ and as $\CC^{2,1}(\overline{\Omega})$ norm of $K(\cdot, X_1)$ is bounded by a deterministic constant $C$, the following CLT in $\CC^{2}(\overline{\Omega})$ (see e.g., \cite[Theorem~3.2]{Naresh.CLT.1976}) holds
 \begin{equation}
     \label{CLT-density}
      \sqrt{n}\left( \widehat{P}_{n,K} - {P}_K \right)\xrightarrow{d} \mathbb{G}.
 \end{equation}
 Then the results follows from \eqref{CLT-density}, \cref{coro:hadamard-otmap} and the functional delta-method \cite[Theorem~3.10.4]{VaartWellner2023}.
 \end{proof}

\section{Proof of~\cref{Theorem:MainWithTime}}\label{Section:Proof of main}
In this section, we provide the proof of Theorem~\ref{Theorem:MainWithTime}. We recall that for each $t$, the function  $\phi_t: \Omega \to \R$ is strictly convex and solves the boundary problem \eqref{eq:MongeAmpere}. Following \cite{urbas1997second}, we rewrite \eqref{eq:MongeAmpere} in a form that allows for a linearization:
\begin{equation}\label{eq:MongeAmpereUrbas}
    \begin{aligned}
           \det( D^2 \phi_t) &= \frac{q}{p_t(\nabla \phi_t)} \quad {\rm in}\,\, \Omega,\\
  h_t(\nabla  \phi_t)  & = 0\quad {\rm in}\,\, \partial \Omega. 
    \end{aligned}
\end{equation}
As indicated previously, the strategy is to differentiate~\eqref{eq:MongeAmpereUrbas} and apply the implicit function theorem \cite[Theorem~15.1]{deimling2013nonlinear} over the functional
\begin{equation}\label{eq: Gamma}
\begin{aligned}
    \Gamma: I \times \CC^{2,\alpha}_{>0}(\bar{\Omega}) &\to \CC^{0,\alpha}(\overline{\Omega}) \times \CC^{1,\alpha}(\partial \Omega)  \\
     \left( \begin{array}{c}
         t  \\
         \phi
    \end{array}\right)&\mapsto \left(\begin{array}{c}
         \Gamma_t^{(1)}(\phi) \\
         \Gamma_t^{(2)}(\phi) \end{array}\right):=\left(\begin{array}{c} 
        \log(\det(D^2\phi))-\log(q)+\log(p_t(\nabla \phi))   \\
          h_t(\nabla \phi)
    \end{array}\right), 
    \end{aligned}
\end{equation}
where for $\lambda \geq 0$, 
$$ \CC^{k,\alpha}_{>\lambda}(\bar{\Omega}):=\{ f\in \CC^{k,\alpha}(\overline{\Omega}): D^2 f> \lambda I\}.$$
We observe that the terms \(\log(p_t(\nabla \phi))\) and \(h_t(\nabla \phi)\) appear in the definition of \(\Gamma\). Therefore, we must ensure that the composition of H\"older continuous functions with the same exponents remains H\"older continuous with exponent unchanged. Unfortunately, such a claim is untrue in general. For example, if $f,g\in \CC^{0,\alpha}(\overline{\Omega})$, then $f\circ g\in \CC^{0,\alpha^2}(\overline{\Omega})$ rather than being $\CC^{0,\alpha}(\overline{\Omega})$. However, the claim is valid when either $f$ or $g$ has higher regularity as in the following lemma, which is an immediate result from \cite[Theorem~4.3] {de1999regularity}.

 \begin{lemma}[Compositions of H\"older Functions]\label{lemma: holder composition}
For any $\alpha,\beta\in[0,1]$ and integers $l\geq k\geq 0$ with $l\geq 1$, there exists a constant $C=C(k,l,\alpha,\beta,\Omega,\Omega')>0$ such that for all $f\in\CC^{k,\alpha}(\Omega')$ and $g\in\CC^{l+1,\beta}(\overline{\Omega})$ satisfying $\nabla g(\overline{\Omega})\subset\Omega'$, we have $f\circ\nabla g\in\CC^{k,\min(\alpha,\beta)}(\overline{\Omega})$ and
\[
\|f\circ\nabla g\|_{\CC^{k,\min(\alpha,\beta)}(\overline{\Omega})}
\le
C\bigl(1+\|g\|_{\CC^{l+1,\beta}(\overline{\Omega})}^{k+\alpha}\bigr)
\|f\|_{\CC^{k,\alpha}(\Omega')}.
\]

\end{lemma}
As a direct consequence of Lemma~\ref{lemma: holder composition}, we establish that $\Gamma$ is well-defined. We conclude this subsection by presenting a set of properties related to H\"older spaces, which will be instrumental in the subsequent subsection for proving that the Fr\'echet derivative of $\Gamma$ with respect to $\phi$ is well-defined. The following lemma, adapted from \cite[Proposition~6.1,6.2]{de1999regularity}, states the stability of H\"older functions under composition, and justifies the condition in \cref{Theorem:MainWithTime}, where we assume that $p_t$ is $\CC^{1,\gamma}$ rather than merely $\CC^{1,\alpha}$.

\begin{lemma}\label{lemma: holder composition stability}
    Under the same setting as in \cref{lemma: holder composition}, the map $$\CC^{k,\alpha}(\Omega')\ni f\mapsto f\circ \nabla g \in \CC^{k,\min{(\alpha,\beta)}}(\overline{\Omega})$$ is linear and continuous. In addition, if $\alpha>\beta$, then the map $$\CC^{l+1,\beta}(
\overline{\Omega})\ni g\mapsto f\circ \nabla g \in \CC^{k,\beta}(\overline{\Omega})$$ is also continuous. More precisely, there exists $\delta,\rho, M>0$ such that if $\norm{g'-g}_{\CC^{l+1,\beta}(\overline{\Omega})}<\delta$, one has that
    \[
\|f\circ \nabla g - f\circ \nabla g'\|_{\CC^{k,\beta}(\overline{\Omega})} \leq M \|f\|_{\CC^{k,\alpha}(\Omega')} \| g -g'\|_{\CC^{l+1,\beta}(\overline{\Omega})}^\rho.
\]
\end{lemma}
The next result is standard and can be found in \cite[Section~4.1]{gilbarg1983trudinger}.
\begin{lemma}[Product of H\"older Functions]\label{lemma: holder product}
For any $\alpha,\beta \in [0,1], l\geq k \geq 0$, there exists a constant $C = C(k,l,\alpha, \beta, {\rm diam}(\Omega))>0$ such that for all $f\in \CC^{k,\alpha}(\overline{\Omega}), g \in \CC^{l,\beta}(\overline{\Omega})$,
    \[
    \norm{fg}_{\CC^{l,\min(\alpha,\beta)}(\overline{\Omega})} \leq C \norm{f}_{\CC^{k,\alpha}(\overline{\Omega})} \norm{g}_{\CC^{l,\beta}(\overline{\Omega})}.
    \]
\end{lemma}
As a consequence of \cref{lemma: holder composition} and \cref{lemma: holder product}, we immediately arrive at the following result as the function $1/g$ may be viewed as the composition of $x\mapsto 1/x$ and $g$, and $x\mapsto 1/x$ is $\CC^\infty$ on $[\lambda, \infty)$.

\begin{lemma}\label{lemma: holder quotient}
 Let $f,g\in \CC^{k,\alpha}(\overline{\Omega})$ and $g \geq c_0$ for some constant $c_0>0$, then $f/g \in \CC^{k,\alpha}(\overline{\Omega})$ and there exists $C = C(k,\alpha,c_0, {\rm diam}(\Omega))$ such that
 \[
 \norm{f/g}_{\CC^{k,\alpha}(\overline{\Omega})} \leq C \norm{f}_{\CC^{k,\alpha}(\overline{\Omega})}(1+\norm{g}_{\CC^{k,\alpha}(\overline{\Omega})}^{k+\alpha}).
 \]
\end{lemma}

\subsection{Fr\'echet Derivative of the Functional}

Now we show that the map $\Gamma$ is $\CC^1$. The proof follows the standard approach used in showing the openness in the continuity method for the Monge--Ampère equation with  Dirichlet boundary conditions (see \cite{figalli2017monge}). Due to the time-varying supports $\Omega_t$, there are technical challenges in ensuring that the functional is always well-defined within its respective spaces, as addressed in the second statement in the following lemma.

\begin{lemma}\label{LemmaFrechet}
 Under the assumptions of Theorem~\ref{Theorem:MainWithTime}, the following holds: 
 \begin{enumerate}[label = (\arabic*)]
 \item $\Gamma_t(\phi_t)=0$ for all $t\in I$. 
     \item For each $t\in I$ there exists an open set $I'\times \CU_t  \subset I\times \CC^{2,\alpha}_{>0}(\overline{\Omega}) $ with $(t,  \phi_t)\in I'\times \CU_t$  such that   $\Gamma  $ is $\CC^1$ in $I'\times \CU_t$.  
     \item For every $t\in I$, the Fréchet derivative of $\Gamma$ at $(t,\phi_t)$ w.r.t. $\phi$ is the linear functional 
     \begin{align*}
         D_\phi\Gamma_t(\phi_t): \CC^{2,\alpha}(\overline{\Omega})  &\to \CC^{0,\alpha}(\overline{\Omega}) \times \CC^{1,\alpha}(\partial \Omega)\\
         \xi &\mapsto \left( \begin{array}{c}
           {\rm tr}([D^2\phi_{t}]^{-1} D^2\xi ) +\frac{ \langle \nabla p_{t}(\nabla \phi_{t}), \nabla \xi\rangle  }{p_{t}(\nabla \phi_{t})}     \\
             \langle \nabla h_{t}(\nabla \phi_{t}), \nabla \xi\rangle  
        \end{array}\right), 
     \end{align*}
     which we  call the \emph{linearized Monge--Ampère operator} at $(t, \phi_t)$. 
     \item For every $t \in I$, the Fréchet derivative of $\Gamma$ at $(t,\phi_t)$ w.r.t.\ $t$ is the element  
     \begin{align*}
         D_t\Gamma_t(\phi_t) 
          = \left( \begin{array}{c}
           \frac{\partial_t p_t(\nabla \phi_t) }{p_t(\nabla\phi_t) }  \\
             \partial_t h_t (\nabla \phi_t) 
        \end{array}\right) \in \CC^{0, \alpha}(\overline{\Omega}) \times \CC^{1,\alpha}(\partial \Omega)\,.
     \end{align*}
 \end{enumerate}
\end{lemma}
\begin{proof}
    The first claim is straightforward by \eqref{eq:MongeAmpere}, $\nabla \phi_t (\partial \Omega)= \partial \Omega_t$, and $h_t(\partial \Omega_t)=\{0\}$.
    We now prove the rest of the claims together. 
    Fix $t\in I$. The boundary regularity provided in \cite{caffarelli1996boundary,urbas1997second} yields the existence of a constant $C=C_t$ such that 
$$ \|\phi_t\|_{\CC^{2, \alpha}(\overline{\Omega})}\leq C \quad {\rm and }\quad  \|\phi_t^*\|_{\CC^{2, \alpha}(\overline{\Omega})}\leq C, $$
which implies that $D^2\phi_t$ belongs to the set of (strictly) positive definite $d\times d$ real matrices over the field $\R$ (see also \cref{coro: Hessian bound} below). Denote such a set as $\mathcal{M}^{+ }_{d\times d}$. Choose $\delta_t>0$ such that $D^2 (\phi_t+ \xi)\in \mathcal{M}^{+ }_{d\times d}$ and $\nabla  (\phi_t+ \xi) (\overline{\Omega})\subset \Omega'$ for all 
$$ \xi \in \mathbb{B}_{\delta_t}:=\left\{f\in \CC^{2, \alpha}(\overline{\Omega}): \| f \|_{\CC^{2, \alpha}(\overline{\Omega})}\leq \delta_t \right\}. $$
(Note that the latter is possible due to the assumption $\bigcup_{t\in I}\Omega_t+\frac{1}{\kappa}\mathbb{B}\subset \Omega'$.)
Set $\CU_t := (\phi_t + \mathbb{B}_{\delta_t})\subseteq \CC^{2,\alpha}_{>0}(\overline{\Omega})$. We show the second claim with the set $I\times \CU_t $. Fix  $(s,\phi) \in I\times \CU_t$. It is  well-known (see, e.g., \cite[Section~3]{figalli2017monge}) that $$\CC^{2,\alpha}_{>0}(\bar{\Omega})\supset \CU_t \ni \phi\mapsto \log \det (D^2 \phi)\in \CC^{0,\alpha}(\overline{\Omega})$$ is $\CC^1$ with directional derivative 
 \[
    \frac{d}{d\vae}\Big|_{\vae =0} \log \det(A+\vae B) = {\rm tr}(A^{-1}B).
    \]
Moreover, as
$$ \Gamma(s, \phi)=\left(\begin{array}{c} 
        -\log(q)+\log(p_s(\nabla \phi))   \\
          h_s(\nabla \phi)
    \end{array}\right)+ \left(\begin{array}{c} 
        \log(\det(D^2\phi))   \\
         0
    \end{array}\right)  $$
   and $\log(q)$  does not vary with  $(s, \phi)$, it suffices to prove the claims for $$(s, \phi)\mapsto (\log(p_s(\nabla \phi)),  
          h_s(\nabla \phi)).$$ 
Fix $\epsilon>0$. The chain rule \cite[Proposition~3.6]{Shapiro1990Directional} and \cref{lemma: holder composition} imply the existence $\delta=\delta(\epsilon)>0$ such that the term $ \phi\mapsto \log(p_t(\nabla \phi)) $ of $\Gamma_t$ satisfies
\begin{equation}\label{eq: log p}
     \left\|\log(p_s(\nabla (\phi+\xi))) - \log(p_s(\nabla \phi)) - \frac{\langle \nabla p_{s}(\nabla \phi), \nabla \xi\rangle }{p_{s}(\nabla \phi)} \right\|_{\CC^{0, \alpha}(\overline{\Omega})}\leq   \vae\|\xi\|_{\CC^{2, \alpha}(\overline{\Omega})}  
\end{equation}
for all $\xi \in \mathbb{B}_{\delta}$.
By the same means, there exists $\delta'=\delta'(\epsilon)>0$ such that the boundary term satisfies 
\begin{equation}\label{eq: log h}
 \left\|h_s(\nabla (\phi+\xi)) - h_s(\nabla \phi) -\langle  \nabla h_s(\nabla \phi), \nabla \xi\rangle \right\|_{  \CC^{1, \alpha}(\partial{\Omega})}\leq  \vae \|\xi\|_{\CC^{2, \alpha}(\overline{\Omega})}
\end{equation}
for all $\xi \in \mathbb{B}_{\delta'}$. Therefore, by choosing $\tilde \delta := \min\{\delta,\delta'\}$,  we derive that for all $\xi \in \mathbb{B}_{\tilde \delta}$,
\[
 \left\| \Gamma(s, \phi+\xi)  - \Gamma(s, \phi)  - D_\phi\Gamma_s(\phi)(\xi) \right\|_{\CC^{2, \alpha}(\overline{\Omega}) \times \CC^{1, \alpha}(\partial{\Omega})}\leq \vae \|\xi\|_{\CC^{0, \alpha}(\overline{\Omega}) }\,.
\]
Similarly, there exists $\delta_0=\delta_0(\epsilon)$ such that for all $|r|<\delta_0$, the following holds
$$ \left\| \Gamma(s+r, \phi)  - \Gamma(s, \phi)  -r\, D_t\Gamma_s(\phi) \right\|_{\CC^{2, \alpha}(\overline{\Omega})}\leq \epsilon |r|   $$
for all $|r|<\delta_0$. Since $\epsilon>0$ was arbitrary,  we conclude that $\Gamma$ is Fréchet differentiable over $I\times \CU_t$. Finally, the  continuity of the derivatives $(s,\phi)\mapsto (D_\phi \Gamma_s(\phi), D_t \Gamma_s(\phi))$ can be easily deduced from Lemmas~\ref{lemma: holder composition}, \ref{lemma: holder composition stability}, \ref{lemma: holder product}, and \ref{lemma: holder quotient}.
\end{proof}

\subsection{Linearized Monge--Amp\`ere Equation}

In this subsection, we start by showing that the linearized Monge--Amp\`ere operator $D_\phi \Gamma_t(\phi_t)$ is a second-order elliptic differential operator (see~\cref{coro: Hessian bound}) with strictly oblique boundary conditions (see~\cref{Lemma: strict obliqueness}). As a consequence, it is invertible (see~\cref{thm: existence and uniqueness linearized MA}) as a mapping from $\CXt$ to $\CYt$, where 
\[
\CXt= \left\{ \xi \in \CC^{2,\alpha}(\overline{\Omega}) : \int_{\partial \Omega_t} p_t \xi(\nabla \phi_t^*) = 0 \right\},
\]
and 
\[
\CYt := \left\{(f,g) \in \CC^{0,\alpha}(\overline{\Omega}) \times \CC^{1,\alpha}(\partial \Omega):  \int_{\Omega} qf = \int_{\partial \Omega_t} p_t g(\nabla \phi_t^*)\right\}.
\]
Here, we recall that $p_t g(\nabla \phi_t^*)$ is the function $x\mapsto p_t(x) g(\nabla \phi_t^*(x))$. 
Here, we equip $\CXt$ with $\CC^{2,\alpha}$ norm and $\CYt$ with the product norm on $\CC^{0,\alpha}(\overline{\Omega})\times \CC^{1,\alpha}(\partial \Omega)$, under which $\CXt,\CYt$ are Banach spaces. We note that $\phi_t \in \CXt$, and, as shown in \cref{thm: existence and uniqueness linearized MA}, $D_t\Gamma(\phi_t)\xi \in \CYt$ for every $\xi\in \CXt$. 
To invert it, we must obtain the existence and uniqueness of solutions in $\CXt$ to the linearized Monge--Amp\`ere equation:
 \begin{align}\label{eq: linearized MongeAmpere general}
   {\rm tr}([D^2\phi_{t}]^{-1} D^2\xi ) +\frac{ \langle \nabla p_{t}(\nabla \phi_{t}), \nabla \xi\rangle  }{p_{t}(\nabla \phi_{t})} & = f\quad \rm in\,\, \Omega,\\
  \langle \nabla h_{t}(\nabla \phi_{t}), \nabla \xi\rangle&=g \quad \rm on \,\, \partial \Omega \notag,
\end{align}
for every $(f,g) \in \CYt$.Due to the absence of a zero-order term, the standard Schauder theory for elliptic equations with oblique derivative boundary conditions~\cite[Section~6.7]{gilbarg1983trudinger} does not apply directly. Motivated by the Fredholm-type argument used in \cite[Theorem~3.1]{nardi2015schauder} to establish solvability of the Neumann problem, and by the strategy of \cite{delanoe1991classical} for solving the second boundary-value problem for the Monge--Amp\`ere equation in the presence of a \emph{non-vanishing} zero-order term when $d=2$, we obtain the following theorem. It establishes solvability of the linearized Monge--Amp\`ere equation and is of independent interest for second-order elliptic equations with strictly oblique boundary conditions and a \emph{vanishing} zero-order term.
 
\begin{theorem}[Solvability of the Linearized Monge--Amp\`ere Equation]\label{thm: existence and uniqueness linearized MA}
 Under the setting of Theorem~\ref{Theorem:MainWithTime}. Let $f \in \CC^{0,\alpha}(\overline{\Omega}), g \in \CC^{1,\alpha}(\partial{\Omega})$.
Then the linearized Monge--Amp\`ere equation \eqref{eq: linearized MongeAmpere general} admits a unique solution $\xi\in \CXt$ if and only if
\begin{equation}\label{eq: compatibility condition}
    \int_{\Omega} qf = \int_{\partial \Omega_t} p_t g(\nabla \phi_t^*). \tag{CC}
\end{equation}
As a consequence, $D_\phi \Gamma_t(\phi_t): \CXt\to \CYt$ is bounded and invertible. 
\end{theorem}
\begin{remark}
    The condition presented in \eqref{eq: compatibility condition} is commonly known as the {\it compatibility condition} in the literature related to the Neumann-type boundary condition.
\end{remark}

We now proceed with the proof of \cref{thm: existence and uniqueness linearized MA}. We start by proving that \eqref{eq: linearized MongeAmpere general} is a second-order elliptic partial differential equation with strictly oblique boundary conditions.  The following lemma controls \( [D^2   \phi_t]^{-1} \) from below and above yielding the uniform ellipticity of $D_\phi \Gamma_t(\phi_t)$.  The proof is omitted as it is presented in \cite[Remark~1.1]{figalli2017monge}.
\begin{lemma}\label{coro: Hessian bound}
Let the setting of Theorem~\ref{Theorem:MainWithTime} hold, for each $t\in I$, there exists  $\beta_t >0$ such that 
\[
\beta_t^{-1} \Id \leq [D^2\phi_t]^{-1} \leq \beta_t \Id\quad \rm in \,\,\overline{\Omega}.
\]
\end{lemma}

The following result shows that  $D_\phi\Gamma_t(\phi_t)$ admits a strictly oblique boundary condition. Its proof follows among the lines of \cite[Section~2]{urbas1997second}.
\begin{lemma}[Strict Obliqueness]\label{Lemma: strict obliqueness}
Under the setting of Theorem~\ref{Theorem:MainWithTime}, for each $t\in I$,  there exist  $\rho_t >0$ such that 
    \[\inner{\nabla h_t(\nabla \phi_t(x))}{\nu(x)}\geq \rho_t\quad \text{for all} \ x\in\partial \Omega,\] 
    where $\nu$ denotes the unit outer normal vector-field to $\partial \Omega$.
\end{lemma}

\begin{proof}
    For brevity, we may drop any sub-indices related to $t$. Define $H(x):= \nabla h(\nabla \phi(x))$ for $ x \in \overline{\Omega}$ and $\chi(x) := \inner{H(x)}{\nu(x)}$ for $x\in \partial \Omega$. Since $\phi \in \CC^{2, \alpha}(\overline{\Omega})$ and $h \in \CC^{2,\alpha}(\overline{\Omega'})$, $\chi$ is continuous on $\partial \Omega$ and admits minimum at some point $x_0\in \partial \Omega$. Let $\{{e}_i\}_{i=1}^d$  be an orthonormal basis of $\R^d$. Up to a translation and rotation, we may assume $x_0= {0}$, ${e}_1,\dots, {e}_{d-1}$ are tangential to $\partial \Omega$ at ${0}$, and $\nu({0}) = e_d$.
    Since $h(\nabla \phi) <0$ in $\Omega$ and equals zero on $\partial \Omega$, we have
    \begin{equation}\label{eq: tangential}
         D_j H({0}) = \inner{DH({0})}{e_j} = \inner{D^2\phi({0}) \nabla h(\nabla \phi({0}))}{e_j} = D_{jk} \phi D_k h = 0\quad {\rm on} \,\,\partial \Omega, 
    \end{equation}
    for $j = 1,\dots, d-1,$
    and
     \begin{equation}\label{eq: normal}
    D_d H({0}) = \inner{DH({0})}{e_d} = \inner{D^2\phi({0}) \nabla h(\nabla \phi({0}))}{e_d} = D_{dk}\phi D_kh \geq 0\quad \rm on \,\,\partial \Omega,
         \end{equation}
    where $D_{dk} \phi := \partial_{e_d}\partial_{{e}_k} \phi({0})$, $D_k h := \partial_{{e}_k} h(D\phi({0}))$, and we are summing over repeated indices. For ease of the notations, the functions are assumed to be evaluated at ${0}$ unless explicitly specified in the rest of this proof. In particular, combining \eqref{eq: tangential} and  \eqref{eq: normal} implies that
    \begin{equation}\label{eq: combined}
            D^2\phi \nabla h = D H,
    \end{equation}
    and 
    \begin{equation}\label{eq: nablaHDnHe}
        D H = D_d H \,  e_d=(0, \dots, D_d H )
    \end{equation}
    where $\nabla h = (D_1h,\dots, D_dh)^\T$.
    Since $\phi$ is strictly convex, $D^2\phi$ is invertible on $\partial \Omega$, and thus
    \begin{equation}\label{eq: Dnh}
            D_d h = e_d^\T [D^2\phi]^{-1} D H.
    \end{equation}   
   Taking the square of both sides of \eqref{eq: Dnh}, we obtain that
    \begin{align*}
            (D_dh)^2 &= (e_d^\T [D^2\phi]^{-1} D H)( e_d^\T [D^2\phi]^{-1} D H) \\
           \eqref{eq: nablaHDnHe}\implies  &= (e_d^\T [D^2\phi]^{-1} (D_dH  \,e_d))( e_d^\T [D^2\phi]^{-1} (D_dH\, e_d)  )\\
         &= e_d^\T [D^2\phi]^{-1}e_d(e_d D_dH)^\T [D^2\phi]^{-1} (D_d He_d)\\
        \eqref{eq: nablaHDnHe}\implies    & = e_d^\T [D^2\phi]^{-1}e_d (D H)^\T [D^2 \phi]^{-1} D H   \\
     \eqref{eq: combined}\implies        & =  e_d^\T [D^2\phi]^{-1}e_d (\nabla h)^\T D^2\phi [D^2 \phi]^{-1} D^2\phi \nabla h\\
            &=  e_d^\T [D^2\phi]^{-1} e_d (\nabla h)^\T D^2\phi \nabla h.
    \end{align*}
    Hence, 
    \begin{equation*}
         \chi({0}) = \inner{\nabla h(\nabla \phi({0}))}{\nu({0})} = D_d h = \sqrt{ e_d^\T [D^2\phi]^{-1}e_d (\nabla h)^\T D^2\phi \nabla h}. 
    \end{equation*}
By applying~\cref{coro: Hessian bound}, we derive the bound
    \begin{equation*}
        \chi({0}) \geq \beta_t^{-2} \sqrt{e_d^\T e_d \|\nabla h\|^2}, 
    \end{equation*}
    which implies the existence of a constant \(\rho_t > 0\) such that $\chi({0}) \geq \rho_t$ as $\|\nabla h\| = 1 $ by assumption. Since $ \inf_{x\in \partial \Omega}\chi(x)=\chi({0})$, the claim follows. 
\end{proof}

 So far, we have shown that the linearized Monge--Amp\`ere operator $D_\phi \Gamma_t(\phi_t)$ is a second-order elliptic differential operator with strictly oblique boundary condition. Moreover, as shown in the following lemma, we may prove that, if a solution exists, it is unique up to a constant shift.

 \begin{lemma}\label{lemma: uniqueness linearized MA}
Under the setting of \cref{Theorem:MainWithTime}, let $\xi\in \CXt$ be a solution to 
\begin{align*}
     {\rm tr}([D^2\phi_{t}]^{-1} D^2\xi ) +\frac{ \langle \nabla p_{t}(\nabla \phi_{t}), \nabla \xi\rangle  }{p_{t}(\nabla \phi_{t})} & = 0\quad \rm in\,\, \Omega,\\
  \langle \nabla h_{t}(\nabla \phi_{t}), \nabla \xi\rangle&=0 \quad \rm on \,\, \partial \Omega.
\end{align*}
Then $\xi=0$. 
\end{lemma}
 \begin{proof}

     Assume that the contrary holds.  By strong maximum principle \cite[Theorem~3.5]{gilbarg1983trudinger}, $\xi$ attains its maximum at some $x_0\in \partial \Omega$.  Let $\nu$ be the outer
normal unit vector field at $\partial \Omega$.  As $\Omega$ is uniformly convex, there exists a ball $B\subset \Omega$ such that $x_0\in \Omega$. \cref{Lemma: strict obliqueness} implies that $\inner{\nabla h_{t}(\nabla \phi_{t}) }{\nu}>0$ on $\partial \Omega$ (hence $\inner{\nabla h_{t}(\nabla \phi_{t}(x_0))}{\nu(x_0)} >0$), so that Hopf's Lemma \cite[Theorem~2.5]{han2011elliptic} yields that \(\inner{\nabla h_{t}(\nabla \phi_{t}(x_0))}{\nabla \xi(x_0)}>0,\)
    which contradicts the boundary condition.
 \end{proof}
Now we are ready to proceed with the proof of Theorem~\ref{thm: existence and uniqueness linearized MA}.  As mentioned earlier, the uniqueness result presented in Lemma~\ref{lemma: uniqueness linearized MA} motivates an application of the Fredholm alternative. It remains to identify the necessary conditions for the solvability of the linearized Monge--Ampère equation for a general \( f \in \CC^{0,\alpha}(\overline{\Omega}) \) and \( g \in \CC^{1,\alpha}(\partial{\Omega}) \).

\begin{proof}[Proof of Theorem~\ref{thm: existence and uniqueness linearized MA}]

We prove necessity and sufficiency separately. 

    \textit{Necessity:}
    Suppose $\xi\in \CXt$ solves \eqref{eq: linearized MongeAmpere general}, then by \eqref{eq:MongeAmpere} 
    \begin{equation}\label{eq: CC 1}
            \int_{\Omega} qf = \int_{\Omega} \left(p_t(\nabla \phi_t)\tr([D^2\phi_t]^{-1}D^2\xi) + \inner{\nabla p_t(\nabla\phi_t)}{\nabla \xi}\right)\det(D^2\phi_t)\,.
    \end{equation}
    Using a change of variable from $x\mapsto \nabla \phi_t(x)$ and noting that $[D^2\phi_t(\nabla \phi_t^*)]^{-1}= D^2 \phi_t^*$, \eqref{eq: CC 1} becomes
    \begin{equation}\label{eq: CC 2}
        \int_{\Omega_t} \tr(D^2\phi_t^* D^2\xi(\nabla \phi_t^*))p_t + \inner{\nabla p_t}{\nabla \xi(\nabla \phi_t^*)} = \int_{\Omega_t} \diver{p_t\nabla \xi(\nabla \phi_t^*)}.
    \end{equation}
    By the divergence theorem \cite[Appendix~C.1]{evans2022partial}, \eqref{eq: CC 2} can be written as
    \[
    \int_{\Omega_t} \diver{p_t\nabla \xi(\nabla \phi_t^*)} = \int_{\partial \Omega_t} \inner{p_t\nabla \xi(\nabla \phi_t^*)}{\nabla h_t} =\int_{\partial \Omega_t } p_t g(\nabla \phi^*_t).
    \]

    \textit{Sufficiency:} By \cref{coro: Hessian bound}, \cref{Lemma: strict obliqueness}, and \cite[Theorem~6.31]{gilbarg1983trudinger}, the operator 
    \begin{align*}
        \mathbb{M}: \CC^{2,\alpha}(\overline{\Omega})& \to \CC^{0,\alpha}(\overline{\Omega})\times  \CC^{1,\alpha}(\partial \Omega )\\
        \xi&\mapsto \left(\begin{array}{c}
             L \\
             B 
        \end{array}\right)= \left(\begin{array}{c}
             {\rm tr}([D^2\phi_t]^{-1} D^2\xi ) +\frac{ \langle \nabla p_t(\nabla \phi_t), \nabla \xi\rangle  }{p_t(\nabla \phi_t)} \\
             \langle \nabla h_t(\nabla \phi_t), \nabla \xi\rangle + \xi
        \end{array}\right)
    \end{align*}
is bounded with bounded inverse $\mathbb{M}^{-1}$. The necessary condition above implies that if $\xi\in \CXt$,  then 
$$  \int_{\Omega} q L(\xi) = \int_{\partial \Omega_t} p_t [B(\xi)-\xi](\nabla \phi_t^*)= \int_{\partial \Omega_t} p_t B(\xi)(\nabla \phi_t^*), $$
so that $\mathbb{M}(\xi)\in \CYt$. As a consequence, $\mathbb{M}\vert_{\CXt}:\CXt\to  \CYt$ is a bijection. Call   $\CL:\CYt \to \CXt$ its (bounded) inverse  and consider the equation 
\begin{equation}\label{eq: Fredholm 1}
    \xi - \CL(0,\xi) = \CL(f,g).
\end{equation}
Here, the boundedness of $\CL$ follows from the inverse mapping theorem \cite[Corollary~2.7]{brezis2011functional}. We see that $\xi \in \CXt$ solves \eqref{eq: linearized MongeAmpere general} if and only if $\xi$ solves \eqref{eq: Fredholm 1}. Note that $\CL(0,\xi)$ is a well-defined element of $\CXt$ as  $(0, \xi)\in \CYt$. 

Now we apply the Fredholm alternative \cite[Theorem~5.3]{gilbarg1983trudinger} to conclude the solvability of \eqref{eq: Fredholm 1}. Define $T:\CXt\ni g \mapsto \CL[0,g]\in  \CXt$. Since $\CL:\CYt \to \CXt$ is bounded and the map
$${\bf i}: \CC^{2,\alpha}(\overline{\Omega})\cap \CXt\ni g \mapsto g \in \CC^{1,\alpha}(\partial \Omega) \cap \left\{ \xi \in \CC^{1,\alpha}(\partial \Omega) : \int_{\partial\Omega_t} p_t \xi(\nabla \phi_t^*) = 0 \right\} $$
is compact, 
the identity 
$ T(g)=\CL(0, {\bf i}(g) )$ yields  the compactness of  $T$. Since \eqref{eq: Fredholm 1} is equivalent to
\begin{equation}\label{eq: Fredholm 2}
    \xi - T\xi = \CL(f,g),
\end{equation}
 the Fredholm alternative implies that \eqref{eq: Fredholm 2} admits a unique solution $\xi \in \CXt$ if and only if the homogeneous equation
\begin{equation}\label{eq: Fredholm 3}
    \xi - T\xi = 0
\end{equation}
has only a trivial solution. As \eqref{eq: Fredholm 3} is equivalent to \eqref{eq: linearized MongeAmpere general},  according to  \cref{lemma: uniqueness linearized MA}, $\xi = 0$ is the unique solution of  \eqref{eq: linearized MongeAmpere general} in $\CXt$. The result follows. 
\end{proof}

\subsection{The Implicit Function Theorem}\label{Sec:implicit}
\cref{thm: existence and uniqueness linearized MA} guarantees the invertibility of \( D_\phi \Gamma_t(\phi_t) \) between the spaces \( \CXt \) and \( \CYt \). However, the range of \( \Gamma \) may be beyond \( \CYt \). Therefore, we modify the functional as follows to satisfy the compatibility conditions \eqref{eq: compatibility condition}. Set $t\in I$ and consider
\begin{align}\label{eq: GammaTilde}
    \tilde{\Gamma}^{(t)}: I \times (\CC^{2,\alpha}_{>0}(\bar{\Omega})\cap \CXt)  &\to \CYt \notag \\
     \left( \begin{array}{c}
         s  \\
         \phi
    \end{array}\right)&\mapsto \left(\begin{array}{c} 
        \Gamma_s^{(1)}(\phi) - \int_{\Omega} q \Gamma_s^{(1)}(\phi) + \int_{\partial \Omega_t} p_t [\Gamma_s^{(2)}(\phi)](\nabla \phi_t^*)    \\
          \Gamma_s^{(2)}(\phi) 
    \end{array}\right). 
\end{align}
The following result shows that for each $t\in I$,  the implicit function theorem (see e.g., \cite[Theorem 17.6]{gilbarg1983trudinger} or \cite[Theorem~15.1 and Corollary~15.1]{deimling2013nonlinear}) can be applied over $\tilde{\Gamma}^{(t)}$.   
%and {\it a fortiori} concludes the proof of Theorem~\ref{Theorem:MainWithTime}. 
\begin{lemma}\label{LemmaFrechetTilde}
Let $t\in I$ be as above. Under the assumptions of \cref{Theorem:MainWithTime}, the following holds: 
 \begin{enumerate}[label = (\arabic*)]
 \item $\tilde{\Gamma}^{(t)}_s(\phi_s)=0$ if and only if $\Gamma_s(\phi_s)=0$. 
     \item There exist an open set $I'\times \CU_t  \subset I\times (\CC^{2,\alpha}_{>0}(\bar{\Omega})\cap \CXt)  $ with $(t,  \phi_t)\in I\times \CU_t$  such that   $\tilde{\Gamma}^{(t)}  $ is $\CC^1$ in $I'\times \CU_t$.  
     \item The Fréchet derivative of $\tilde{\Gamma}^{(t)}$ w.r.t. $\phi\in \CXt$ at $(t,\phi_t)\in I\times (\CC^{2,\alpha}_{>0}(\bar{\Omega})\cap \CXt) $, namely   $D_\phi\tilde \Gamma_t(\phi_t) $,  satisfies $D_\phi\tilde \Gamma_t(\phi_t) h= D_\phi  \Gamma_t(\phi_t)h $ for all $h\in \CXt$. 
     \item The bounded linear operator $D_\phi\tilde \Gamma_t(\phi_t) :\CXt\to \CYt$ is invertible.   
 \end{enumerate}

\end{lemma}
\begin{proof}
    The second point is a direct consequence of \cref{LemmaFrechet} and the fourth point holds due to the third point and \cref{thm: existence and uniqueness linearized MA}. 

    {\it Proof of (1).} We prove that if $\phi\in \CC^{2,\alpha}_{>0}(\bar{\Omega})$ solves 
    $$ \det( D^2 \phi) = C\frac{q}{p_s(\nabla \phi)} \quad {\rm in}\,\, \Omega,\ 
  h_s(\nabla  \phi)  = 0\quad {\rm in}\,\, \partial \Omega\notag, $$
    for some $C\in \R$, then $C=1$. Note that $ h_s(\nabla  \phi)  = 0$ in $\partial \Omega$ implies $ \nabla  \phi(\Omega)  = \Omega_s$, then 
  $$ C=\int_{\Omega} \det( D^2 \phi)  p_s(\nabla \phi) = \int_{\Omega_s} p_s=1  $$
  and the claim follows. 

  {\it Proof of (3).} From the linearity of integration, one have for $\xi\in \CXt$
\[
D_\phi \tilde{\Gamma}^{(t)}_t(\phi_t)(\xi) = \left( \begin{array}{c}
    D_\phi \Gamma_t^{(1)}(\phi_t)(\xi) - \int_{\Omega} q D_\phi \Gamma_t^{(1)}(\phi_t)(\xi) + \int_{\partial \Omega_t} p_t[D_\phi \Gamma_t^{(2)}(\phi_t)(\xi)](\nabla \phi_t^*)   \\
       D_\phi \Gamma_{t}^{(2)}(\phi_t)(\xi)
\end{array} \right). 
\]
In view of $D_\phi \Gamma_t(\phi_t)(\CXt)=\CYt$ by \cref{thm: existence and uniqueness linearized MA}, 
$$ \int_{\Omega} q D_\phi \Gamma_t^{(1)}(\phi_t)(\xi)= \int_{\partial \Omega_t} p_t[D_\phi\Gamma_t^{(2)}(\phi_t)(\xi)](\nabla \phi_t^*),  $$
then the result follows.
\end{proof}
The implicit function theorem in Banach spaces (see, e.g.,  \cite[Theorem~15.1]{deimling2013nonlinear}) implies that  $I \ni t\mapsto \phi_t\in \CC^{2, \alpha}(\overline{\Omega})$ is of class  $\CC^1$, meaning that $\{\phi_t\}_{t\in I}$ is a $\CC^1$ curve. 
Moreover, it holds that
\[
[D_\phi \tilde \Gamma_t^{(t)}(\phi_t)](\partial_t \phi_t) = -D_t \tilde \Gamma_t^{(t)}(\phi_t),
\]
so that $\partial_t \phi_t$ is the unique solution of 
\begin{align}\label{LinearizedMAProofTosimplify}
   {\rm tr}([D^2\phi_{t}]^{-1} D^2\xi ) +\frac{ \langle \nabla p_{t}(\nabla \phi_{t}), \nabla \xi\rangle  }{p_{t}(\nabla \phi_{t})} & = -\frac{\partial_t p_t(\nabla \phi_t)}{p_t(\nabla \phi_t)} -\int_{\Omega} q\frac{\partial_t p_t(\nabla \phi_t) }{p_t(\nabla\phi_t) } +\int_{\partial \Omega_t} p_t {\partial_t h_t }  \quad \rm in\,\, \Omega,\\
  \langle \nabla h_t(\nabla \phi_{t}), \nabla \xi\rangle&=-\partial_t h_t(\nabla \phi_t) \quad \rm on \,\, \partial \Omega.\notag
\end{align} 
We focus now on simplifying the right-hand side of \eqref{LinearizedMAProofTosimplify}. To do so, we prove 
\begin{equation}\label{CompatibilityOflinearizedRight}
   \int_{\Omega_t} q \partial p_t =\int_{\partial \Omega_t} p_t {\partial_t h_t },  
\end{equation}
which follows as a consequence of the following Reynolds transport theorem (see e.g., \cite[Appendix~C.4]{evans2022partial}) for implicitly defined surfaces. 
\begin{lemma}
    Let $\{h_t\}_{t\in I}$, $\{\Omega_t\}_{t\in I}$ and $\Omega$ be as in \cref{Theorem:MainWithTime}. Then for every $\CC^1$ curve $\{f_t\}_{t\in I}$ on $\CC(\Omega')$ it holds that 
    $$ \partial_t \int_{\Omega_t} f_t =  \int_{\Omega'} \partial_t f_t - \int_{\partial\Omega_t} f_t \partial_t h_t .  $$
\end{lemma}
\begin{proof}
Set $t\in I$ and 
    define the curve $T_s=(\nabla h_{t+s})^{-1} \nabla h_{t} $. By assumption and by the inverse function theorem, $\{T_s\}_{s\in I}\subset \CC(\Omega')$ is $ \CC^1$,   $h_{t+s}(T_s)=0$ in $\partial\Omega_t$, and $T_s(\Omega_t)=\Omega_s$. Applying the second-order Taylor development on   $h_{t+s}$ provides the existence of a curve $\{\omega_s\}_{s\in I}$ on $\CC(\partial \Omega_t)$ such that $o(\|\omega_s \|_{\CC(\partial \Omega_t)} ) =o(s)$ and 
    $$ h_t= 0= h_{t+s}(T_s)=h_{t+s}+ \langle \nabla h_{t+s}, T_s- T_0  \rangle + \omega_s\quad \text{on}\ \partial \Omega_t.  $$
    As a consequence, we get the relation 
    \begin{equation}
        \label{eq:Lastequality}
        \partial_t h_t = -\langle \left.\partial_s \right|_{s=0} T_s, \nabla h_{t}  \rangle \quad \text{on}\ \partial \Omega_t.
    \end{equation}
    Since $\nabla h_t$ is the unit outer-normal vector field to $\partial \Omega_t$, Reynolds transport theorem implies that 
    \begin{align*}
        \partial_t \int_{\Omega_t} f_t &=  \int_{\Omega'} \partial_t f_t + \int_{\partial\Omega_t} f_t  \left\langle  \left.\partial_s \right|_{s=0} T_s, \nabla h_t \right\rangle 
    \end{align*}
and the result follows by \eqref{eq:Lastequality}. 
\end{proof}

\section*{Acknowledgments}
The authors thank Eustasio del Barrio, Nicolás García Trillos, Alejandro Garriz Molina, Tudor Manole, Gilles Mordant and Marcel Nutz for helpful comments.

\bibliographystyle{plain}
 \bibliography{ref_no_url}
 
  \end{document}